\documentclass[a4paper,abstracton]{scrartcl}
\usepackage[utf8]{inputenc}
\usepackage{amsfonts,amsthm,amssymb,amsmath}
\usepackage[inline]{enumitem}
\usepackage{color}
\usepackage{graphicx} 
\usepackage{authblk}

\renewcommand{\leq}{\leqslant}
\renewcommand{\geq}{\geqslant}
\newtheorem{thm}{Theorem}
\newtheorem{lem}[thm]{Lemma}

\theoremstyle{remark}

\definecolor{pink}{RGB}{219, 48, 122}

\title{A note on the 1-2-3 Theorem for infinite graphs}
\author{Marcin Stawiski \\ stawiski@agh.edu.pl}

\affil{AGH University,\\ Faculty of Applied Mathematics, \protect\\al. Mickiewicza 30, 30-059 Krakow, Poland}

\begin{document}

\maketitle
\begin{abstract}

Karoński, Łuczak and Thomason conjectured in 2004 that for every finite graph without isolated edge, the edges can be assigned weights from  $\{1,2,3\}$ in such a way that the endvertices of each edge have different sums of incident edge weights. This is known as the 1-2-3 Conjecture, and it was only recently proved by Keusch. We extend this result to infinite graphs in the following way. If $G$ is a graph without isolated edge, then the edges can be assigned weights from  $\{1,2,3\}$ is such a way that the endvertices of each edge have different sum of incident edge weights, or these endvertices have both the same infinite degree. We also investigate the extensions of theorems about total and list versions of 1-2-3 Conjecture to infinite graphs.


\end{abstract}

\section{Introduction}

Let $G=(V,E)$ be a finite or infinite graph. Let $k$ be a natural number. A \emph{k-edge weighting} is a function $\omega\colon E \rightarrow \{1, \dots, k \}$. The \emph{weighted degree} of a vertex $v$ of $G$ is the sum $s_\omega(v)=\Sigma_{w\in N(v)}\omega(vw)$. We say that $\omega$ is a k-\emph{weighting} of $G$ if it is a k-edge weighting such that  the endvertices of each edge have different weighted degrees. Karoński, Łuczak, and Thomason \cite{KLT} in 2004 conjectured that each finite graph without isolated edge  has a 3-weighting. This is known as the 1-2-3 Conjecture. This conjecture has gained a lot of attention since it was formulated. There are many partial result regarding this conjecture and similar problems. The 1-2-3 Conjecture was only recently proved in its full generality by Keusch.

\begin{thm}[Keusch \cite{keusch}]\label{thm:Keusch}Let $G$ be a finite graph without isolated edge, then there exists a $3$-weighting of $G$.
\end{thm}

In this note we are interested in extending theorem of Kusch to infinite graphs in the most general way possible. Notice that the weighted degree of a vertex of infinite degree is always equal to its degree. Hence, we cannot obtain different sums between vertices of the same infinite degree. The main result of this paper is the following theorem.

\begin{thm}\label{thm}
    Let $G$ be a graph without isolated edge, then there exists a $3$-weighting of $G$.
\end{thm}

The 1-2-3 Conjecture was also studied in different variants. In particular, it was studied for total weightings \cite{przybylowozniak12, wongzhu2}, and in the list version \cite{wongzhu2,wongzhu1,zhu1}.   The most general of these versions is for total list weightings. Let $L_V=\{L_v:v\in V\}$ be a set of lists of weights for vertices of $G$ such that $L(v)\subseteq\{1,\dots,k\}$, and $L_E=\{L_e:e\in E\}$ be a set of lists of weights for edges such that $L(e)\subseteq\{1,\dots, k\}$, and $L=L_V \cup L_E$. We say that a function $\omega\colon V\cup E\rightarrow \{1,\dots, l\}$ is a weighting of $G$ from a set of lists $L$ if $\omega(v)\in L_v$ for every vertex $v$ of $G$, and $\omega(e)\in L_e$ for every edge $e$ of $G$. We define the \emph{weighted degree} of a vertex $v$ by $s_\omega(v)=\Sigma_{w\in N(v)}\omega(vw)+\omega(v)$. We say that a graph $G$ has a $L$-weighting if there exists a weighting $\omega$ of $G$  from a set of lists $L$ such that for every edge of $G$, each of its endvertices have different weighted degree, or both its endvertices have the same infinite degree. This generalizes the previously defined notions, as we can obtain the edge versions by simply setting each $L_v=\{1\}$.
The general theorem proved in this paper allows us to obtain extensions of theorems about lists and total weightings. In particular, we prove the following theorems.
%\begin{thm}\label{thm:total}
%    Let $G$ be a graph. Then there exists a weighting of vertices from $\{1,2\}$ and edges from $\{1,2,3\}$ such that for every edge, both its vertices have different total weighted degree.
%\end{thm}

\begin{thm}\label{thm:list}
    Let $G$ be a graph without isolated edge, and $L$ be a set of lists for $G$ such that each list for a vertex has one element, and each list for an edge has at least five elements. Then there exists a $L$-weighting of $G$.
\end{thm}

\begin{thm}\label{thm:list1}
    Let $G$ be a graph, and $L$ be a set of lists for $G$ such that each list for a vertex has at least two elements, and each list for an edge has at least three elements. Then there exists a $L$-weighting of $G$.
\end{thm}

Theorems \ref{thm}, \ref{thm:list}, and \ref{thm:list1} shall follow from more general Theorem \ref{thm:general} proved in the next section, and their respective versions for finite graphs.
\begin{thm}\label{thm:general}Let $G$ be a graph without isolated edge, and $L$ be a list of weights for $G$. Let $H$ be a subgraph of $G$. Let $L_H$ be the set of lists obtained from $L$ by restriction to the set $V(H)\cup E(H)$. If each finite induced subgraph $H$ of $G$ not isomorphic to $K_2$ has a $L_H$-weighting, then $G$ has a $L$-weighting.
\end{thm}

For more on the 1-2-3 Conjecture and related problems see the survey of Seamone \cite{seamone}.

\section{General result}

Before we prove Theorem \ref{thm:general} in its full generality, we prove it first for locally finite graphs.
\begin{lem}\label{lem}Let $G$ be a locally finite graph without isolated edge, and $L$ be a list of weights for $G$. Let $H$ be a subgraph of $G$. Let $L_H$ be the set of lists obtained from $L$ by restriction to the set $V(H)\cup E(H)$. If each finite induced subgraph $H$ of $G$ not isomorphic to $K_2$ has a $L_H$-weighting, then $G$ has a $L$-weighting.
\end{lem}
\begin{proof}
By assumption if $G$ is finite, then there exists a $L$-weighting of $G$. Therefore, we can assume that $G$ is infinite. Let $H$ be a component of $G$. It is enough to prove that there exists a $L_H$-weighting of each such $H$.

Choose an arbitrary vertex $v$ of $H$. Let $B(v,i)$ be the ball at centre $v$ and radius $i$. Define $H_i=H[B(v,i)]$ for each $i\geq 2$. If $2\leq j <i$, and $\omega_i$ is a $L_{H_i}$-weighting, then define $\omega_{i,j}'=\omega_i|_{H_{j}}$. Define $\Omega_j'=\{\omega_{i,j}'\colon \omega_i\textnormal{ is a }L_{H_i}\textnormal{-weighting of } H_i \textnormal{ for some } i\geq 2\}$. Note that $\omega_{j}'\in \Omega_j'$ is not necessarily a $L_{H_j}$-weighting of $H_{j}$. By assumption  $\Omega_j'$ for each $j\geq 2$ is non-empty. For each $2\leq k<j$ if $\omega_k\in \Omega_k'$, then there exists a $\omega'_j \in \Omega_j'$ such that $\omega_k' \subset \omega_j'$. Each set $\Omega_i'$ is finite, therefore by K\H{o}nig's Lemma there exists a function weighting $\omega$ of $H$ from the set of lists $L_H$ such that for each $i \geq 2$ we have $\omega|_{H_i}\in \Omega_j'$. We will prove that $\omega$ is a $L_H$-weighting of $H$.

Take any edge $tu \in E(H)$. We prove that the weighted degrees of $t$ and $u$ in $H$ is different. There exists $i\geq 2$ such that $t,u \in B(v,i)=V(H_i)$. Let $\omega'=\omega|_{V(H_{i+1}}$. For each $z \in V(H_i)$ the sum $s_{\omega}(z)=s_{\omega'}(z)$ because each addend is the same in both sums. It follows that $s_\omega(u)=s_{\omega'}(u)\neq s_{\omega}(t)=s_{\omega'}(t)$.
\end{proof}

After proving the lemma above, we are ready to prove  Theorem \ref{thm:general}.

\begin{proof}
First, define $G'$ as the graph obtained from $G$ by deleting all the edges between vertices of infinite degree. Notice that if there exists a $L_H$-weighting of each component $H$ of $G'$ which is not isomorphic to $K_2$, then there exists a $L$-weighting of $G$. Let $H$ be a component of $G'$ which is not isomorphic to $K_2$. It is enough to prove that each such $H$ has a $L_H$-weighting.

We construct a graph $H'$ in the following way. Take a vertex $v$ of infinite degree $\kappa$. Let $\{u_i \colon i\leq \kappa\}$ be the set of neighbours of $v$. We replace the vertex $v$ with vertices $v_i$ for $i \leq \kappa$, and we add an edge between vertices $v_i$ and $u_i$ for each $i \leq \kappa$.  We set $L(v_i)=L(v)$. We repeat this procedure for each vertex $v$ of infinite degree. We denote the new set of lists by $L'$. We claim that if there exists a $L'$-weighting of each component of $H'$ not isomorphic to $K_2$, then there exists a 1-2-3 weighting of $H$. 
%A $L'$-weighting of each component of $H'$ not isomorphic to $K_2$ exists by assumption because each vertex of $H'$ has finite degree.

If a component $F$ of $H'$ is isomorphic to $K_2$, then one of its vertices $u$ have degree one in $H$, and the other is obtained by splitting some vertex $v$ of infinite degree. Vertex $v$ is a unique neighbour of $u$ in $H$. Therefore, the sum in $u$ is always different that the sum in $v$ in $H$. If $F$ is not isomorphic to $K_2$, then 
by Lemma \ref{lem} there exists a $L'_F$-weighting $\omega_F$. We define the weighting $\omega$ of $H$ in a following way. If $uw$ is an edge of $H$, then $\omega(uw)=\omega_F(uw)$ for some component $F$ of $H'$ which contains both $u$ and $v$. If $u$ is a vertex of $H$ and $w$ is obtained by splitting the vertex $v$ of $H$, then we put $\omega(uv)=\omega_F(uw)$. If both $u$ and $w$ are obtained by splitting the vertices $v$ and $t$, then we can define $\omega(vt)$ arbitrarily. We will show that the weighting $\omega$ is $L_H$-weighting of $H$.

Let $uw$ be an edge in $H$. If both $u$ and $w$ have finite degree, then $s(u)=s_F(u)\neq s_F(w)=s_F(w)$ for some component $F$ which contains both $u$ and $v$. If $u$ or $w$ has infinite degree, then the sum of $s(w)$ or $s(u)$ is infinite. It remains to notice that if both $w$ and $u$ have infinite but distinct degrees, then their sum is equal to their degrees, and therefore it is distinct.
\end{proof}

\section{Application of the general result}

 Theorem \ref{thm} is obtained from Theorem \ref{thm:general} by setting  $L(v)=\{1\}$ for each vertex $v$, and by setting $L(e)=\{1,2,3\}$. The assumptions of Theorem \ref{thm:general} are now satisfied by Theorem \ref{thm:Keusch}. To prove Theorem \ref{thm:list} we need its finite version.

%\begin{thm}[Kalkowski \cite{kalkowski12}]\label{thm:totalfinite} Let $G$ be a finite graph. Then there exists a weighting of vertices from $\{1,2\}$ and edges from $\{1,2,3\}$ such that for every edge, both its vertices have different total weighted degree.
%\end{thm}

%Theorem \ref{thm:total} is obtained from Theorem \ref{thm:general} by setting each list $L(v)=\{1\}$ for each vertex $v$. The assumptions of Theorem \ref{thm:general} are now satisfied by Theorem \ref{thm:totalfinite}. The list version of 

\begin{thm}[Zhu \cite{zhu1}]\label{thm:listfinite}    Let $G$ be a finite graph without isolated edge, and $L$ be a set of lists for $G$ such that each list for a vertex has one element, and each list for an edge has at least five elements. Then there exists a $L$-weighting of $G$.
\end{thm}
The statement of the theorem above is conjectured \cite{wongzhu1} to hold even if we restrict ourselves to lists for edges with only three elements.
Theorem \ref{thm:list} is obtained from Theorem \ref{thm:general} as the assumptions of Theorem \ref{thm:general} are satisfied by Theorem \ref{thm:list}.

\begin{thm}[Wong, Zhu \cite{wongzhu2}]\label{thm:listfinite1}
    Let $G$ be a finite graph, and $L$ be a set of lists for $G$ such that each list for a vertex has at least two elements, and each list for an edge has at least three elements. Then there exists a $L$-weighting of $G$.
\end{thm}

Again, Theorem \ref{thm:list1} is obtained from Theorem \ref{thm:general} because the assumptions of Theorem \ref{thm:general} are satisfied by Theorem \ref{thm:listfinite1}.

\bibliographystyle{abbrv}
\bibliography{sources.bib}

\end{document}